\documentclass[12pt,a4paper]{amsart}

\usepackage{latexsym} 
 
\usepackage{graphicx}
\usepackage{epsfig}
\usepackage{mathrsfs}
\usepackage{amssymb}
\usepackage{amsthm}
\usepackage{amsfonts}
\usepackage{amsmath}
\usepackage{centernot}
\usepackage{amstext}
\usepackage{amscd}
\usepackage{enumerate}
\usepackage{tikz}
\usepackage{epic}
\usepackage{eepic}
\usepackage{pstricks,pst-plot}
\usepackage{bigdelim}
\usepackage{multirow}

\numberwithin{equation}{section}

\theoremstyle{plain}
\newtheorem{thm}{Theorem}[section]

\newtheorem{lem}[thm]{Lemma}
\newtheorem{ass}[thm]{Assertion}

\newtheorem{theorem*}{Theorem}[]

\theoremstyle{definition}
\newtheorem{defn}[thm]{Definition}

\newtheorem{rem}{Remark}

\theoremstyle{remark}

\theoremstyle{property}

\newcommand{\N}{\mathbb{N}}
\newcommand{\R}{\mathbb{R}}

\DeclareMathOperator{\grad}{grad\,}

\def\accentclass@{7}
\def\makeacc@#1#2{\def#1{\mathaccent"\accentclass@#2 }}
\makeacc@\cir{017}

\def\dis{\displaystyle}

\def\min{\mathop{\rm min}}

\title[Correction : The Kuo condition and an inequality of Thom's type]
{Correction to : The Kuo condition, an inequality \\ 
of Thom's type and ($C$)-regularity}

\author{Karim Bekka and Satoshi Koike}

\address{Institut de recherche Mathematique de Rennes, 
Universit\'{e} de Rennes 1, Campus Beaulieu, 35042 Rennes cedex, France}
\address{Department of Mathematics, Hyogo University of Teacher Education,
942-1 Shimokume, Kato, Hyogo 673-1494, Japan}
\email{karim.bekka@univ-rennes1.fr}
\email{koike@hyogo-u.ac.jp} 



\thanks{This research is partially supported by  
JSPS KAKENHI Grant Number 20K03611. }

\date{\today}

\begin{document}


\maketitle

\begin{abstract} We correct a statement of a theorem on characterisation of the $(c)$-regularity we gave in Topology 37 (1998), 45--62.
This theorem was used in the paper in the proof of two theorems on the $(c)$-regular stratification.
In this note we give a weaker version of the theorem as an alternative lemma which ensures the $(c)$-regularity condition 
and turn to be sufficient for the proof of the theorems.
\end{abstract}

\bigskip

\section{Introduction}
We gave a characterisation of the $(c)$-regularity in Theorem 2.4 of \cite{bekkakoike}. 
In order to show the theorem, we used a statement that 
a one-dimensional subspace of $\R^{n+p}$ 
$$
K_x := (Ker \ d\rho (x) \cap T_x X)^{\bot} \cap T_x X 
$$
is orthogonal to $\R^n (:= \R^n \times \{ 0\}$) at any $x \in X$.  
But this statement does not necessarily hold, namely $K_x$ is not always orthogonal to $\R^n$ 
as T. Gaffney pointed out to us in \cite{gaffney}. 
Therefore the proof of Theorem 2.4 in \cite{bekkakoike} is false.
In addition, we proved in \cite{bekkakoike} Theorem 2.7 (resp. Theorem 2.8) that 
the stratification $\Sigma (\R^n \times J)$ (see \S 3 for the definition)  
is $(c)$-regular under the Kuo condition (resp. the second Kuo condition), using Theorem 2.4. 

In this note we show an alternative lemma to Theorem 2.4, adding one more 
assumption to the original two assumptions.
The $(c)$-regularity follows from these three assumptions. 
In addition, we show that the Kuo condition (resp. the second Kuo condition) implies not only 
the original assumptions but the new assumtion of the lemma. 
In other words, Theorem 2.7 (resp. Theorem 2.8) in \cite{bekkakoike} follows from the lemma.

The authors would like to thank T. Gaffney for pointing out to us a mistake 
in the proof of Theorem 2.4 in \cite{bekkakoike}.


\section{Alternative lemma}\label{lemma}

Let $M$ be a smooth manifold, and let $X, \ Y$ be smooth submanifolds of $M$ such that
$Y \subset \overline{X}$. 
We suppose now that $M$ is endowed with a Riemannian metric. 
Let $(T_Y , \pi , \rho )$ be a smooth tubular neighbourhood for $Y$ with the associated 
projection $\pi$ and a smooth non-negative control function $\rho$ such that 
$\rho^{-1}(0) = Y$ and $\grad \rho (x) \in Ker \ d\pi (x)$. 

In this section we treat a lemma on regularity conditions in the stratification theory.
Let us recall some regularity conditions.

\begin{defn}
(1) We say that the pair $(X,Y)$ is {\em Whitney $(a)$-regular at} $y_0 \in Y$, 
if for any sequence of points $\{ x_i\}$ of $X$ which tends to $y_0$ such that
the sequence of tangent spaces
$\{ T_{x_i}X\}$ tends to some plane $\sigma$ in the Grassman space of 
$\dim X$-planes, then we have $T_{y_0}Y \subset \sigma$. 
We say that $(X,Y)$ is {\em $(a)$-regular} if it is $(a)$-regular at any point 
$y_0 \in Y$.
(See \cite{whitney1, whitney2} for properties of the Whitney $(a)$-regularity.)

(2) We say that the pair $(X,Y)$ is {\em $(c)$-regular at $y_0 \in Y$
for the control function} $\rho$, 
if for any sequence of points $\{ x_i\}$ of $X$ which tends to $y_0$ such that
the sequence of planes
$\{ Ker \ d\rho (x_i) \cap T_{x_i}X\}$ tends to some plane $\tau$ in the 
Grassman space of $(\dim X - 1)$-planes, then we have $T_{y_0}Y \subset \tau$. 
We say that $(X,Y)$ is {\em $(c)$-regular for the control function} $\rho$ 
if it is $(c)$-regular at any point $y_0 \in Y$ for the function $\rho$.
(See \cite{bekka} for properties of the $(c)$-regularity.)
\end{defn} 

\begin{defn}
We say that the pair $(X,Y)$ satisifies {\em condition $(m)$}, 
if there exists some positive number $\epsilon > 0$ such that 
$(\pi ,\rho )|_{X \cap T^{\epsilon}_Y} : X \cap T^{\epsilon}_Y \to Y \times \R$ is a submersion 
where $T^{\epsilon}_Y := \{ x \in T_Y \ | \ \rho (x) < \epsilon \}$. 
\end {defn}

Let $y_0 \in Y$, and let $\{ x_i\}$ be an arbitrary sequence of points of $X$ which tends to $y_0$ 
such that the sequence of planes $\{ Ker \ d\rho (x_i ) \cap T_{x_i} X\}$ tends to some plane $\tau$ 
in the Grassmann space of $(\dim X - 1)$-planes. 
We call such a sequence of points of $X$ {\em pre-regular}.
Taking a subsequence of $\{ x_i\}$ if necessary,  we may assume that the sequence of planes 
$\{ Ker \ d\rho (x_i )\}$ and $\{ T_{x_i} X\}$ tend to some planes $\mu$ and $\sigma$ 
in the Grassmann spaces of $(\dim M -1)$-planes and $( \dim X)$-planes, respectively.
We say that the pair $(X,Y)$ satisfies {\em condition $(c_d)$ at $y_0$}, if 
$\dim (\mu \cap \sigma) = \dim X - 1$ for any pre-regular sequence 
of points of $X$ tending to $y_0$. 
In addition, we say that the pair $(X,Y)$ satisfies {\em condition $(c_d)$}, 
if $(X,Y)$ satisfies condition $(c_d)$ at any point $y_0 \in Y$.  

\begin{rem}\label{remark1}
Let $M = \R^{n+p}$ and $Y = \R^n \times \{ 0\}$. 
Then $\dim (\mu \cap \sigma ) = \dim X - 1$ if and only if 
$\mu$ and $\sigma$ are transverse at $y_0$ 
if and only if $\sigma \not\subset \mu$ 
if and only if $\mu^{\perp} \not\subset {\sigma}^{\perp}$. 
\end{rem}

We show the following lemma, which gives a sufficient condition for the $(c)$-regularity.

\begin{lem}\label{Lemma}
Suppose that the pair $(X,Y)$ is $(a)$-regular at $y_0 \in Y$ and satisfies condition $(m)$ 
for a control function $\rho$ and condition $(c_d)$ at $y_0 \in Y$.
Then $(X,Y)$ is $(c)$-regular at $y_0 \in Y$ for the function $\rho$.
\end{lem}

\begin{proof}
We first recall the setting of the proof of Theorem 2.4 in \cite{bekkakoike}. 
There is a chart $\Phi : (U, U \cap Y, y_0) \to (\R^{n+p}, \R^n \times \{ 0 \} , 0)$  for $Y$
at $y_0$ such that 

\vspace{3mm}

(1) $\Phi \circ \pi \circ \Phi^{-1}$ is the projection $(y,t) \mapsto (y,0)$ from
$\R^{n+p} \to \R^n \times \{ 0\}$. 

(2) $\grad (\rho \circ \Phi^{-1} )(y,t) \in \{ 0\} \times \R^p$, i.e. is orthogonal 
to $\R^n \times \{ 0 \}$. 

\vspace{3mm}

\noindent It follows that it is enough to prove the lemma in the case where 
$M = T^{\epsilon}_Y = \R^{n+p}$ and $Y = \R^n \times \{ 0\}$, 
$X$ is Whitney $(a)$-regular over $\R^n \times \{ 0\}$ at $0 \in \R^{n+p}$, 
$\pi : \R^n \times \R^p \to \R^n$ defined by $\pi (y,t) = t$ and 
$\rho : \R^{n+p} \to \R$ satisfy 

\vspace{3mm}

\begin{itemize}
\item $\rho^{-1}(0) = Y$,

\item $(\pi ,\rho )|_X : X \to \R^n \times \R$ is a submersion, 
 
\item $\grad \rho (y,t) \in \{ 0\} \times \R^p$ 
if $(y,t) \in \R^{n+p} \setminus \R^n \times \{ 0\}$, 
\end{itemize}

\vspace{3mm}

\noindent and $(X,Y)$ satisfies condition $(c_d)$ at $0 \in \R^{n+p}$. 

Let us remark that condition $(m)$ guarantees 
$$
\dim (Ker \ d\rho (x) \cap T_x(X)) = \dim X - 1
$$ 
at any point $x \in X$. 

Let $\{ x_i\}$ be a sequence of points of $X$ which tends to $0 \in \R^{n+p}$ 
such that the sequence of planes $\{ Ker \ d\rho (x_i ) \cap T_{x_i} X\}$ tends to some plane $\tau$ 
in the Grassmann space of $(\dim X - 1)$-planes. 
Taking a subsequence of $\{ x_i\}$ if necessary,  we may assume that the sequences of planes 
$\{ Ker \ d\rho (x_i )\}$ and $\{ T_{x_i} X\}$ tend to some planes $\mu$ and $\sigma$ 
in the Grassmann spaces of $(n + p -1)$-planes and $( \dim X)$-planes, respectively.
Note that $\R^n \times \{ 0\} \subset Ker \ d\rho (x_i)$ for any $i \in \N$.
Therefore we have
$$
\mu = \lim_{i \rightarrow \infty} Ker \ d\rho (x_i) \supset \R^n \times \{ 0\}.
$$
By the Whitney ($a$)-regularity, we have
$$
\sigma = \lim_{i \rightarrow \infty} T_{x_i}X \supset \R^n \times \{ 0\}.
$$
Therefore we have $\mu \cap \sigma \supset \R^n \times \{ 0\} .$

On the other hand, $\tau =  \lim_{i \rightarrow \infty} (Ker \ d\rho (x_i) \cap T_{x_i}X)$. 
It follows that $\tau \subseteq \mu \cap \sigma$. 
Since $(X,Y)$ satisfies condition $(c_d)$ at $0 \in \R^{n+p}$, 
we have $\dim (\mu \cap \sigma ) = \dim X - 1$. 
Therefore we have $\tau = \mu \cap \sigma \supset \R^n \times \{ 0\}$. 
Thus $(X,Y)$ is $(c)$-regular at $0 \in Y$ for the function $\rho$.
\end{proof}

\bigskip

\section{Proofs of Theorems 2.7, 2.8 in \cite{bekkakoike}}\label{proof}

Let $\mathcal{E}_{[s]} (n,p)$, $n \ge p$,  denote the set of $C^s$ map-germs $: (\R^n,0) \to (\R^p,0)$, 
and let $j^r f(0)$ denote the $r$-jet of $f$ at $0 \in \R^n$ for $f \in \mathcal{E}_{[s]}(n,p)$, 
$s \ge r$.
For $f \in \mathcal{E}_{[r]}(n,p)$, $\mathcal{H}_r(f;\overline{w})$ denotes the {\em  
horn-neighbourhood of $f^{-1}(0)$ of degree $r$ and width} $\overline{w}$,
$$
\mathcal{H}_r(f;\overline{w}) := \{ x\in \R^n : \ |f(x)| \le \overline{w} |x|^r \} .
$$
Let $v_1, \cdots , v_p$ be $p$ vectors in $\R^n$ where $n \ge p$.
The {\em Kuo distance $\kappa$} (\cite{kuo}) is defined by
$\dis
\kappa(v_1, \ldots,v_p) = \displaystyle \min_{i}\{\text{distance of }\, 
v_i\, \text{ to }\, V_i\},
$
where $V_i$ is the span of the $ v_j$'s, $j\ne i$.
In the case where $p = 1$, $\kappa (v) = \| v \| .$

\begin{defn}\label{Kuo condition} 
A map-germ $f \in \mathcal{E}_{[r]}(n,p)$ satisfies the {\em Kuo condition}, 
if there are positive numbers $C$, $\alpha$, $\overline{w} > 0$ such that 
$$
\kappa (\grad f_1 (x), \cdots , \grad f_p (x)) \ge C |x|^{r-1}
$$
in $\mathcal{H}_r (f;\overline{w}) \cap \{ |x| < \alpha \}$. 
\end{defn}

\begin{defn}\label{second Kuo condition} 
A map-germ $f \in \mathcal{E}_{[r+1]}(n,p)$ satisfies the {\em second Kuo condition}, 
if for any map $g \in \mathcal{E}_{[r+1]}(n,p)$ with $j^{r+1}g(0) = j^{r+1}f(0)$,
there are positive numbers $C$, $\alpha$, $\overline{w}, \delta > 0$ 
(depending on $g$) such that 
$$
\kappa (\grad f_1 (x), \cdots , \grad f_p (x)) \ge C |x|^{r-\delta}
$$
in $\mathcal{H}_{r+1} (g;\overline{w}) \cap \{ |x| < \alpha \}$. 
\end{defn}

Let us recall Theorems 2.7 and 2.8 in \cite{bekkakoike}. 
Let $f : (\R^n,0) \to (\R^p,0)$, $n \ge p$, be a $C^r$ (resp. $C^{r+1}$) map, 
and let $J$ be a bounded open interval containing $[0, 1]$. 
For arbitrary $g \in \mathcal{E}_{[r]}(n,p)$ (resp. $\mathcal{E}_{[r+1]} (n,p)$) 
with $j^r g(0) = j^r f(0)$, define a $C^r$ (resp. $C^{r+1}$) map 
$F : (\R^n \times J, \{ 0\} \times J) \to (\R^p,0)$ by 
$F(x,t) := f(x) + t(g(x) - f(x))$. 
Let us remark that the Kuo condition (resp. the second Kuo condition) guarantees 
that $F^{-1}(0) \setminus \{ 0\} \times J$ is smooth around $\{ 0\} \times J$ 
if it is not empty.
Therefore, if $F^{-1}(0) \ne \{ 0\} \times J$ as set-germs at $\{ 0\} \times J$,
$$
\Sigma (\R^n \times J) := \{ \R^n \times J \setminus F^{-1}(0), 
F^{-1}(0) \setminus \{ 0\} \times J, \{ 0\} \times J \}
$$ 
gives a stratification of $\R^n \times J$ around $\{ 0\} \times J$ 
under the assumption of the Kuo condition (resp. the second Kuo condition).  
In this case, $\dim (F^{-1}(0) \setminus \{ 0\} \times J) = n +1 - p$. 

If $F^{-1}(0) = \{ 0\} \times J$ as set-germs at $\{ 0\} \times J$,
$$
\Sigma (\R^n \times J) := \{ \R^n \times J \setminus \{ 0\} \times J, \{ 0\} \times J \}
$$ 
gives a stratification of $\R^n \times J$ around $\{ 0\} \times J$.

\begin{thm}\label{thm2.7} (\cite{bekkakoike}, Theorem 2.7)
If a $C^r$ map $f \in \mathcal{E}_r (n,p)$ satisfies the Kuo condition, then the stratification
$\Sigma (\R^n \times J)$ is $(c)$-regular.
\end{thm}

\begin{thm}\label{thm2.8} (\cite{bekkakoike}, Theorem 2.8)
Let $f \in \mathcal{E}_{[r+1]}(n,p)$. 
If, for any polynomial map $h$ of degree $r + 1$ such that $j^r h(0) = j^r f(0)$,
there are positive numbers $C$, $\alpha$, $\overline{w}, \delta > 0$ 
(depending on $h$) such that 
$$
\kappa (\grad f_1 (x), \cdots , \grad f_p (x)) \ge C |x|^{r-\delta}
$$
in $\mathcal{H}_{r+1} (h;\overline{w}) \cap \{ |x| < \alpha \}$, 
then the stratification $\Sigma (\R^n \times J)$ is $(c)$-regular.
\end{thm}

\begin{rem}\label{remark21} 
The condition which $f \in \mathcal{E}_{r+1} (n,p)$ in Theorem \ref{thm2.8} 
satisfies is equivalent to the second Kuo condition.
\end{rem}

\begin{proof}[Proof of Theorem \ref{thm2.7}]
Let us show Theorem \ref{thm2.7}, using Lemma \ref{Lemma}.  
In the case where $F^{-1}(0) = \{ 0\} \times J$ as set-germs at $\{ 0\} \times J$, 
it is obvious that $\Sigma (\R^n \times J)$ is a $(c)$-regular stratification.
Therefore we consider only the case where $F^{-1}(0) \ne \{ 0\} \times J$ 
as set-germs at $\{ 0\} \times J$.
We set $X := \R^n \setminus F^{-1}(0)$, $Y := F^{-1}(0) \setminus \{ 0\} \times J$ 
and $Z := \{ 0\} \times J$. 
Then we can easily see that the pairs $(X,Y)$ and $(X,Z)$ are $(c)$-regular (around $Z$).
In order to show that the pair $(Y,Z)$ is $(c)$-regular, we have to check 
the $(a)$-regularity, condition $(m)$ and condition $(c_d)$. 
Here the control function is a non-negative function $\hat{\rho} : \R^n \times \R \to \R$ 
defined by $\hat{\rho}(x,t) := x_1^2 + \cdots + x_n^2$. 
We can show the $(a)$-regularity and condition $(m)$ similarly  
to the proof of Theorem 2.7 in \cite{bekkakoike}.  
Therefore it remains to show that the pair $(Y,Z)$ satisfies condition $(c_d)$. 

For $t \in J$, define a $C^r$ map $f_t : (\R^n ,0) \to (\R^p ,0)$ by $f_t(x) := F(x,t)$, 
namely $f_t(x) := f(x) + t(g(x) - f(x))$. 

The $r$-jet of $f$ at $0 \in \R^n$, $j^r f(0)$, has a unique polynomial representative $z$  
of degree not exceeding $r$.
We do not distinguish the $r$-jet $j^r f(0)$ and the polynomial representative $z$ here. 
We set $q(x) := f(x) - z(x)$ and $r(x) := g(x) - z(x)$, and define 
$P_t (x) := q(x) + t(r(x) - q(x))$, $t \in J$. 
Then $f_t (x) = z(x) + P_t(x)$ where $P_t : (\R^n,0) \to (\R^p,0)$ be a $C^r$ map 
such that $j^r P_t(0) = 0$ for $t \in J$. 

\begin{rem}\label{remark3}
(1) For $a _1> 0$, there are positive numbers $b_1$, $\beta_1 > 0$ with $0 < b_1 \le a_1$ 
such that 
$$
\mathcal{H}_r (f;b_1) \cap \{ |x| < \beta_1 \} \subset \mathcal{H}_r (f_t;a_1) \cap \{ |x| < \beta_1 \}
$$
for any $t \in J$.

(2)  For $a _2> 0$, there are positive numbers $b_2$, $\beta_2 > 0$ with $0 < b_2 \le a_2$ 
such that 
$$
\mathcal{H}_r (f;;b_2) \cap \{ |x| < \beta_2 \} \subset \mathcal{H}_r (z;a_2) \cap \{ |x| < \beta_2 \} . 
$$
\end{rem}

We denote by $V_x$ (resp. $V_{t,x}$) the $p$-dimensional subspace of $\R^n$ spanned by 
$$
\{ \grad z_1 (x), \cdots , \grad z_p (x)\} \ \ (\text{resp.} \ \{ \grad f_{t,1}(x), \cdots , \grad f_{t,p}(x)\} )
$$
for $x \in \mathcal{H}_r(f;\overline{w}) \cap \{ |x| < \alpha \}$ and $t \in J$. 
Concerning $V_x$, we proved the following property in \cite{bekkakoike}.

\begin{ass}\label{claimI} (\cite{bekkakoike}, Claim I) 
Let $\epsilon_1$ be an arbitrary positive number. 
Then there are positive numbers $\alpha_1$, $\overline{w}_1$ with 
$0 < \alpha_1 \le \alpha$ and $0 < \overline{w}_1 \le \overline{w}$ such that 
$$
d(x,V_x) \ge (1 - \epsilon_1)|x| \ \ in \ \ 
\mathcal{H}_r(z;\overline{w}_1) \cap \{ |x| < \alpha_1\} . 
$$
\end{ass}

We first recall the proof of Claim III in \cite{bekkakoike} 
(since the details are not mentioned in the paper). 
Let us denote by $v(x)$ and $v_t(x)$ the projections of $x$ on $V_x$ and $V_{t,x}$, 
respectively. 
For $x \in \mathcal{H}_r(f;\overline{w} ) \cap \{ |x| < \alpha \}$ (and $t \in J$),  
we consider $\{ N_1 (x), \cdots , N_p(x)\}$ the basis of $V_x$ constructed 
as follows:
$$
N_j (x) := \grad z_j(x) - \tilde{N}_j(x) \ \ (1 \le j \le p),
$$
where $\tilde{N}_j(x)$ is the projection of $\grad z_j(x)$ to the subspace 
$V_{x,j}$ spanned by the $\grad z_k(x)$, $k \ne j$, and 
$\{ N_{t,1} (x), \cdots , N_{t,p} (x)\}$ the corresponding basis of $V_{t,x}$. 
Then we have the following assertion.

\begin{ass}\label{claimIV} (\cite{bekkakoike}, Claim IV) 
For any $\epsilon_2 > 0$, there are positive numbers 
$\alpha_2$, $\overline{w}_2$ with $0 < \alpha_2 \le \alpha$ and 
$0 < \overline{w}_2 \le \overline{w}$ such that the following 
inequality holds:
$$
(1 + \epsilon_2 )|N_j (x)| \ge |N_{t,j} (x)| \ge (1 - \epsilon_2)|N_j(x)| \ \ 
(1 \le j \le p)
$$
for $x \in \mathcal{H}_r(f;\overline{w}_2 ) \cap \{ |x| < \alpha_2 \}$ and $t \in J$.
\end{ass}

\noindent By Lemma 3.2 in T.-C. Kuo \cite{kuo}, we have 
$$
v(x) = \sum_{j=1}^p <x, \grad z_j (x)> \frac{N_j(x)}{|N_j(x)|^2}, \ \ 
v_t(x) = \sum_{j=1}^p <x, \grad f_{t,j} (x)> \frac{N_{t,j}(x)}{|N_{t,j}(x)|^2} \ \ 
$$
for $x \in \mathcal{H}_r(f;\overline{w} ) \cap \{ |x| < \alpha \}$ and $t \in J$.
From the proof of Claim I in \cite{bekkakoike}, we can assume that for any $\epsilon_3 > 0$, 
there are positive numbers $\alpha_3$, $\overline{w}_3$ with $0 < \alpha_3 \le \alpha$ 
and $0 < \overline{w}_3 \le \overline{w}$ such that 
\begin{equation}\label{|v|}
|v(x)| \le \sum_{j=1}^p \frac{|<x, \grad z_j (x)>|}{|N_j (x)|} \le \epsilon_3 |x| \ \ 
\text{in} \ \ \mathcal{H}_r(z;\overline{w}_3 ) \cap \{ |x| < \alpha_3\} . 
\end{equation}
Since $f_t(x) = z(x) + P_t(x)$, $t \in J$, there are positive numbers 
$\alpha_4$, $\overline{w}_4$ with
$0 < \alpha_4 \le \min \{ \alpha_2 , \alpha_3\}$ and 
$0 <\overline{w}_4 \le \min  \{ \overline{w}_2 , \overline{w}_3\}$
such that for $x \in \mathcal{H}_r(f;\overline{w}_4 ) \cap \{ |x| < \alpha_4\}$ and $t \in J$, 
the following inequalities hold:
\begin{eqnarray*}
& |v_t(x)| & \le \sum_{j=1}^p \frac{|<x, \grad f_{t,j} (x)>|}{|N_{t,j} (x)|} \\
& & \le \sum_{j=1}^p \frac{|<x, \grad z_j (x)>|}{|N_{t,j} (x)|} + 
\sum_{j=1}^p \frac{|<x, \grad P_{t,j} (x)>|}{|N_{t,j} (x)|} \\
& & \le \frac{\epsilon_3}{1 - \epsilon_2} |x| + \sum_{j=1}^p \frac{|<x, \grad P_{t,j} (x)>|}{|N_{t,j} (x)|}.
\end{eqnarray*}
Note that 
$$
j^{r-1}(\frac{\partial P_{t,j}}{\partial x_i}) (0) = 0 \ \ (1 \le i \le n, \ 1 \le j \le p) \ 
\text{for} \ t \in J.
$$
Therefore, for any $\epsilon_4 > 0$, there are positive numbers $\alpha_5$, 
$\overline{w}_5$ with $0 < \alpha_5 \le \alpha_4$ and $0 < \overline{w}_5 \le \overline{w}_4$  
such that
\begin{equation}\label{|v_t|}
|v_t (x)| \le \epsilon_4 |x| \ \ \text{for} \ x \in \mathcal{H}_r(f;\overline{w}_5 ) \cap \{ |x| < \alpha_5\}
\ \text{and} \ t \in J
\end{equation}
under the assumption of the Kuo condition. 
Then, by (\ref{|v|}) and (\ref{|v_t|}), we have the following assertion. 

\begin{ass}\label{claimIII} (\cite{bekkakoike}, Claim III) 
For any $\epsilon_5 > 0$, there are positive numbers $\alpha_6$, $\overline{w}_6$ 
with $0 < \alpha_6 \le \alpha_5$ and $0 < \overline{w}_6 \le \overline{w}_5$ 
such that
$$
|d(x,V_{t,x}) - d(x,V_x)| \le |v_t (x) - v(x)| \le |v_t(x)| + |v(x)| \le \epsilon_5 |x|
$$
for $x \in \mathcal{H}_r(f;\overline{w}_6 ) \cap \{ |x| < \alpha_6\}$ and $t \in J$. 
\end{ass}

The most important result of \S 3 in \cite{bekkakoike} follows from 
Assertions \ref{claimI} and \ref{claimIII}.

\begin{lem}\label{claimII} (\cite{bekkakoike}, Claim II)
There are positive numbers $\alpha_0$, $\overline{w}_0$ 
with $0 < \alpha_0 \le \min \{ \alpha_1 , \alpha_6\}$ and 
$0 < \overline{w}_0 \le \min \{ \overline{w}_1 , \overline{w}_6\}$ such that
$$
d(x,V_{t,x}) \ge \frac{1}{2} |x| \ \ 
\text{for} \ x \in \mathcal{H}_r(f;\overline{w}_0 ) \cap \{ |x| < \alpha_0 \} \ \text{and} \ t \in J. 
$$ 
\end{lem}

We next denote by $W_{(x,t)}$ the $p$-dimensional subspace of $\R^n \times \R$ spanned 
by 
$$
\{ \grad F_1 (x,t), \cdots , \grad F_p (x,t)\} \ \ \text{for} \   
x \in \mathcal{H}_r(f;\overline{w}) \cap \{ |x| < \alpha \} \ \text{and} \ t \in J.
$$
Then we can show the following lemma.

\begin{lem}\label{c_d} There are positive numbers $\alpha_{11}$, $\overline{w}_{11}$ 
such that 
$$
d((x,0),W_{(x,t)}) \ge \frac{1}{4} |(x,0)| \ \  \text{for} \ 
x \in \mathcal{H}_r(f;\overline{w}_{11} ) \cap \{ |x| < \alpha_{11} \} \ \text{and} \ t \in J.
$$
\end{lem}

\begin{proof}
Let us remark that 
$$
\grad F_j (x,t) = (\grad f_{t,j}(x), \frac{\partial F_j}{\partial t}(x,t)), \ \ 1 \le j \le p.
$$
We denote by $U_{(x,t)}$ the $p$-dimensional subspace of $\R^n \times \R$ spanned 
by 
$$
\{ (\grad f_{t,1}(x), 0),  \cdots , (\grad f_{t,p}(x), 0)\} \ \ \text{for} \   
x \in \mathcal{H}_r(f;\overline{w}) \cap \{ |x| < \alpha \}  \ \text{and} \ t \in J.
$$  
By Lemma \ref{claimII}, we have
\begin{equation}\label{1/2}
d((x,0), U_{(x,t)}) = d(x,V_{t,x}) \ge \frac{1}{2} |x| \  
\text{for} \ x \in \mathcal{H}_r(f;\overline{w}_0 ) \cap \{ |x| < \alpha_0 \} \ 
\text{and} \ t \in J. 
\end{equation}

Let $u(x,t)$ and $\omega (x,t)$ be the projections of $(x,0)$ on $U_{(x,t)}$ 
and $W_{(x,t)}$, respectively. 
Then we have
$$
d((x,0),U_{(x,t)}) = |(x,0) - u(x,t)|, \ \ d((x,0),W_{(x,t)}) = |(x,0) - \omega (x,t)|
$$
for  $x \in \mathcal{H}_r(f;\overline{w} ) \cap \{ |x| < \alpha \}$ and $t \in J$. 
Therefore we have 
\begin{equation}\label{ineq}
|d((x,0),U_{(x,t)}) - d((x,0),W_{(x,t)})| \le |u(x,t) - \omega (x,t)|
\le |u(x,t)| + |\omega (x,t)|
\end{equation}
for  $x \in \mathcal{H}_r(f;\overline{w} ) \cap \{ |x| < \alpha \}$ and $t \in J$. 

For $x \in \mathcal{H}_r(f;\overline{w} ) \cap \{ |x| < \alpha \}$ and $t \in J$, 
let us consider $\{ M_1 (x,t), \cdots , M_p(x,t)\}$ the basis of $U_{(x,t)}$ constructed 
as follows:
$$
M_j (x,t) := (\grad f_{t,j}(x),0) - \tilde{M}_j(x,t) \ \ (1 \le j \le p),
$$
where $\tilde{M}_j(x,t)$ is the projection of $(\grad f_{t,j}(x),0)$ to the subspace 
$U_{(x,t),j}$ spanned by the $(\grad f_{t,k}(x),0)$, $k \ne j$, and let 
$\{ L_1 (x,t), \cdots , L_p (x,t)\}$ be the corresponding basis of $W_{(x,t)}$. 
By definition, we have
$$
|\grad F_j (x,t) - (\grad f_{t,j}(x), 0)| = |\frac{\partial F_j}{\partial t}(x,t)| 
= |g_j (x) - f_j(x)|
$$
where  $j^r (g_j - f_j)(0) = 0$ $(1 \le j \le p)$. 
Therefore there are positive numbers $\alpha_7$, $\overline{w}_7$ with 
$0 < \alpha_7 \le \alpha$ and $0 < \overline{w}_7 \le \overline{w}$ 
such that for any $\lambda_j$, 
$$
\frac{|\sum_j \lambda_j (\grad F_j(x,t) - (\grad f_{t,j}(x),0))|} 
{|\sum_j \lambda_j (\grad f_{t,j}(x),0)|} \rightarrow 0 \ \ \
\text{as} \ x \to 0
$$ 
in $x \in \mathcal{H}_r(f;\overline{w}_7) \cap \{|x| < \alpha_7\}$ 
(uniformly for $t \in J$) under the assumption of the Kuo condition.
Then, using a similar argument to the proof of Claim IV in \cite{bekkakoike}
(see Assertion \ref{claimIV} above), 
we can show the following assertion.

\begin{ass}\label{est}
For any $\epsilon_6 > 0$, there are positive numbers 
$\alpha_8$, $\overline{w}_8$ with $0 < \alpha_8 \le \alpha_7$ and 
$0 < \overline{w}_8 \le \overline{w}_7$ such that the following 
inequality holds:
$$
(1 + \epsilon_6 )|M_j (x,t)| \ge |L_j (x,t)| \ge (1 - \epsilon_6)|M_j(x,t)| \ \ 
(1 \le j \le p)
$$
for $x \in \mathcal{H}_r(f;\overline{w}_8 ) \cap \{ |x| < \alpha_8 \}$ and $t \in J$.
\end{ass}

For $x \in \mathcal{H}_r(f;\overline{w} ) \cap \{ |x| < \alpha \}$ and $t \in J$, 
$$
u(x,t) = \sum_{j=1}^p <(x,0), (\grad f_{t,j} (x),0)> \frac{M_j (x,t)}{|M_j (x,t)|^2},
$$
$$
\omega (x,t) = \sum_{j=1}^p <(x,0), \grad F_j (x,t)> \frac{L_j (x,t)}{|L_j (x,t)|^2}. 
$$
By construction, $<(x,0), (\grad f_{t,j}(x),0)> = <x, \grad f_{t,j}(x)>$ and 
$|M_j (x,t)| = |N_{t,j}(x)|$, $j = 1, \cdots , p$, for $t \in J$. 
It follows from (\ref{|v_t|}) that 
\begin{equation}\label{|u(x,t)|}
|u(x,t)| \le  \sum_{j=1}^p \frac{|<(x,0), (\grad f_{t,j}(x), 0)>|}{|M_j (x,t)|} 
\le \epsilon_4 |x|
\end{equation}
for $x \in \mathcal{H}_r(f;\overline{w}_5 ) \cap \{ |x| < \alpha_5\}$ and $t \in J$.
On the other hand,
$$
|\omega (x,t)| \le \sum_{j=1}^p \frac{|<(x,0), \grad F(x,t)>|}{|L_j (x,t)|} 
= \sum_{j=1}^p \frac{|<(x,0), (\grad f_{t,j}(x), 0)>|}{|L_j (x,t)|} 
$$
for $x \in \mathcal{H}_r(f;\overline{w} ) \cap \{ |x| < \alpha \}$ and $t \in J$. 
By Assertion \ref{est}, we have 
\begin{equation}\label{|w(x,t)|}
|\omega (x,t)| \le \frac{\epsilon_4}{1 - \epsilon_6} |x| \ \ 
\text{for} \ x \in \mathcal{H}_r(f;\overline{w}_9 ) \cap \{ |x| < \alpha_9\} \ \text{and} \ t \in J, 
\end{equation}
where $\alpha_9 = \min \{ \alpha_5 , \alpha_8 \}$ and 
$\overline{w}_9 = \min \{ \overline{w}_5, \overline{w}_8 \}$. 

By (\ref{ineq}), (\ref{|u(x,t)|}) and (\ref{|w(x,t)|}), we have the following assertion.

\begin{ass}\label{III}
For any $\epsilon_7 > 0$, there are positive numbers 
$\alpha_{10}$, $\overline{w}_{10}$ with $0 < \alpha_{10} \le \alpha_9$ and 
$0 < \overline{w}_{10} \le \overline{w}_9$ such that 
$$
|d((x,0),U_{(x,t)}) - d((x,0).W_{(x,t)})| \le \epsilon_7 |x| \ \ 
\text{for} \ x \in \mathcal{H}_r(f;\overline{w}_{10}) \cap \{ |x| < \alpha_{10}\} \ \text{and} \ t \in J.
$$
\end{ass} 

By (\ref{1/2}) and Assertion \ref{III}, there are positive numbers $\alpha_{11}$, $\overline{w}_{11}$ 
{\em with} $0 < \alpha_{11} \le \alpha_{10}$ and $0 < \overline{w}_{11} \le \overline{w}_{10}$  
such that 
$$
d((x,0),W_{(x,t)}) \ge \frac{1}{4} |(x,0)| \ \  \text{for} \ 
x \in \mathcal{H}_r(f;\overline{w}_{11} ) \cap \{ |x| < \alpha_{11} \} \ \text{and} \ t \in J.
$$

\end{proof}

By Lemma \ref{c_d}, we have
$$
d(\frac{(x,0)}{|(x,0)|},W_{(x,t)}) \ge \frac{1}{4} \ \ \text{for} \ x \in 
\mathcal{H}_r(f;\overline{w}_{11} ) \cap \{ 0 < |x| < \alpha_{11} \} \ \text{and} \ t \in J.
$$
Let $\ell_{(t,x)}$ be the the $1$-dimensional subspace of $\R^n \times \R$ 
spanned by $\grad \hat{\rho} (x,t)$ for $x \ne 0$ and $t \in \R$. 
Here $\grad \hat{\rho} (x,t) = (2x_1, \cdots , 2x_n , 0)$.  
Therefore we have 
\begin{equation}\label{keyestimation}
\overline{d} (\ell_{(x,t)},W_{(x,t)}) \ge \frac{1}{4} \ \ \text{for} \ x \in 
\mathcal{H}_r(f;\overline{w}_{11} ) \cap \{ 0 < |x| < \alpha_{11} \} \ \text{and} \ t \in J.
\end{equation}
Here 
$$
\overline{d}(\ell ,W) := \max_{||v|| = 1} \{ d(v,W) \ | \ v \in \ell \}
$$
for subspaces $\ell$ , $W$ of $\R^m$ with $\dim \ell \le \dim W$. 
Note that $\ell \subset W$ if and only if $\overline{d}(\ell , W) = 0$. 

Let $(0,t_0) \in Z = \{  0\} \times J$, and let $\{ (x_i,t_i)\}$ be any pre-regular 
sequence of points of $Y = F^{-1}(0) \setminus \{ 0 \} \times J$ which tends to $(0,t_0)$. 
Namely, the sequence of planes $\{ Ker \grad  \hat{\rho} ((x_i,t_i) ) \cap T_{(x_i,t_i)} Y\}$ 
tends to some plane in the Grassmann space of $(n - p)$-planes. 
Taking a subsequence of $\{ (x_i,t_i)\}$ if necessary,  we may assume that the sequence of planes 
$\{ Ker \grad \hat{\rho} ((x_i,t_i) )\}$ and $\{ T_{(x_i,t_i)} Y\}$ tend to some planes $\mu$ and 
$\sigma$ in the Grassmann spaces of $n$-planes and $(n + 1 - p)$-planes, respectively. 
By (\ref{keyestimation}), we have 
$$
\overline{d} (\ell_{(x_i,t_i)},W_{(x_i,t_i)}) \ge \frac{1}{4} \ \ \text{for} \ (x_i,t_i) \in 
Y \cap \{ 0 < |x| < \alpha_{11} \} .
$$
Since $W_{(x_i,t_i)} = (T_{(x_i,t_i)}Y)^{\perp}$ and $\ell_{(x_i,t_i)} = (Ker \grad \hat{\rho}((x_i,t_i)))^{\perp}$, 
it follows that $\mu^{\perp} \not\subset \sigma^{\perp}$.
By Remark 1, this implies that $(Y,Z)$ is $(c_d)$-regular at $(0,t_0)$.

This completes the proof of Theorem \ref{thm2.7}.
\end{proof}

The proof of Theorem \ref{thm2.8} goes almost in the same way as the above argument.


\end{document}